\documentclass[12pt,a4paper]{article}

\usepackage[a4paper,top=24mm,bottom=24mm,left=25mm,right=25mm]{geometry}
\usepackage{amsmath,amssymb,amsthm,mathtools,bm}
\usepackage{booktabs,tabularx,array,longtable}
\usepackage{xcolor}
\usepackage{enumitem}
\usepackage{fancyhdr}
\usepackage{hyperref}
\usepackage[nameinlink,noabbrev]{cleveref}
\usepackage{microtype}

\setlist{nosep,leftmargin=2.2em}
\allowdisplaybreaks[3]
\numberwithin{equation}{section}

\renewcommand{\headrulewidth}{0.4pt}
\newtheorem{theorem}{Theorem}[section]
\newtheorem{lemma}[theorem]{Lemma}

\newtheorem{remark}[theorem]{Remark}
\newtheorem{definition}[theorem]{Definition}

\newcommand{\R}{\mathbb{R}}
\newcommand{\eps}{\varepsilon}
\newcommand{\dd}{\,\mathrm{d}}
\newcommand{\esssup}{\operatorname*{ess\,sup}}
\newcommand{\avg}[1]{\overline{#1}}
\newcommand{\pdim}{\overline{\dim}_{\mathrm B,p}}

\newcommand{\ep}{\varepsilon}

\newcommand{\norm}[1]{{\left\Vert #1 \right\Vert}}

\begin{document}
\thispagestyle{empty} \thispagestyle{fancy} {
	\fancyhead[RO,LE]{\scriptsize \bf }
	\fancyfoot[CE,CO]{}}
\renewcommand{\headrulewidth}{0pt}
\begin{center}
	{\bf \Large An $\eps$-regularity criterion at one scale for  3-D Navier-Stokes equations and improved upper bound on the box dimension of the interior singular set}
	
	\vspace{4mm}
	
	{\small Yi Lu  \quad  Rongrong Chi  \quad Wei Ren\footnote{Corresponding author, email:
			renwei4321@163.com.cn}
		
		{\small  Hebei Key Laboratory of Machine Learning and Computational Intelligence, College of Mathematics and Information Science, Hebei
			University, Baoding, 071002, PR China.}}
\end{center}

\vspace{3mm}

\begin{center}
	\begin{minipage}{140mm}
{\bf \small Abstract}. 
In this paper, we obtain an  $\varepsilon$-regularity criterion at one scale for 3D Navier-Stokes equations using a coarse-to-refined estimate argument based on mean-fluctuation decomposition.
Making use of this  $\varepsilon$-regularity criterion, we improve the upper bound on the box dimension of the interior singular set to  $15/13$.


\end{minipage}
\end{center}

\vspace{5mm}

\section{Introduction}
Let $\Omega$ be a domain in $\R^3$ and let $T>0$. We consider the
three-dimensional incompressible Navier--Stokes equations
\begin{equation}\label{eq:NS}
	\begin{cases}
		&\partial_tu-\Delta u+(u\cdot\nabla)u+\nabla\pi=0,
		\text{in } \Omega_T:=\Omega\times(-T,0),\\
		&\nabla\cdot u=0, \\
		&u|_{t=0}=u_0.
	\end{cases}
\end{equation}
Here $u: \Omega_T\to\R^3$ is the velocity and $\pi$ is the
pressure. A point $z_0=(x_0,t_0)$ is called regular if $u$ is essentially
bounded in some parabolic neighborhood of $z_0$. The remaining interior points
form the interior singular set denoted by $\mathcal{S}$. Denote $B_r(x_0)=\{y\in \mathbb{R}^{3}:|x_0-y|\leq r\}$, $B_r=B_r(0)$, $Q_r(z_0)=B_r(x_0)\times(t_0-r^2,t_0)$, \ $I_r=(-r^2,0)$, $Q(r)=B_r\times I_r$, $L^{p,q}(Q(r))=L^p(-r^2,0; L^q(B_r)), \ L^{p}(Q(r))=L^{p,p}(Q(r))$.

The study of partial regularity for suitable weak solutions to 3D Navier-Stokes equations was initialed from Scheffer \cite{[Scheffer1]}-\cite{[Scheffer3]}.
In a celebrated work \cite{[CKN]},  Caffarelli, Kohn and Nirenberg obtained two $\ep$-regularity criteria to  suitable weak solutions of \eqref{eq:NS}: $z=0$ is a regular point provided that one of the following conditions holds for an absolute positive constant $\ep$,
\begin{equation}
	\label{CKN}		
	\norm{u}_{L^{3}(Q(1))}+\|u\pi\|_{L^1(Q(1))}+\|\pi\|_{L^{1,5/4}(Q(1))} < \ep.
\end{equation}
and
\begin{equation}
	\label{ckn2}
	\limsup_{r\rightarrow0}  r^{-1/2}\norm{\nabla u}_{L^{2}(Q(r))} < \ep.
\end{equation}
Lin \cite{[Lin]}, Ladyzhenskaya and Seregin \cite{[LS]} gave an alternative condition
\begin{equation}\label{Lin}
	\norm{u}_{L^{3}(Q(1))}+\|\pi\|_{L^{3/2}(Q(1))} < \ep.
\end{equation}
 Vasseur \cite{[Vasseur]}, further obtained for any $p>1$,
\begin{equation}\label{Vasssmall}
\|u\|_{L^{\infty, 2}(Q(1))} + \|\nabla u\|_{L^{2}(Q(1))}  +\|\Pi\|_{L^{p, 1}(Q(1))} <\varepsilon.
\end{equation}
 In this direction, Wang and Zhang \cite{[WZ]} obtained
\begin{equation}\label{WZ}
\|u\|_{L^{\infty,2}(Q(1))} +\| u\|_{L^{2, 4} (Q(1))}+\|\Pi\|_{L^{1,2}(Q(1))} < \varepsilon. 
\end{equation}
Choi   and Vasseur \cite{[CV]}, further derived
\begin{equation}\label{CVGP}
\|u\|_{L^{\infty,2}(Q(1))} + \|\nabla u\|_{L^{2}(Q(1))}  +\|\Pi\|_{L^{1}(Q(1))} < \varepsilon.
\end{equation}
 Guevara and Phuc \cite{[GP]}  improved criteria to
\begin{equation}\label{GP}
	\norm{u}_{L^{2p, 2q} (Q(1))}+\|\pi\|_{L^{p, q}(Q(1))} < \ep,
\end{equation}
for some $p, q$ satisfying ${2}/{p}+{3}/{q} = 7/2$ and $1\leq q\leq2$.
Subsequently, He, Wang and Zhou \cite{[HWZ]} found  that \eqref{GP} can be replaced with
\begin{equation}
	\label{optical}
	\norm{u}_{L^{p,q}(Q(1))}+\|\pi\|_{L^{1}(Q(1))}<\ep
\end{equation}
for some $p, q$ satisfying $1\leq 2/p+3/q <2$ and $1\leq p,\,q\leq\infty$.
Since the gradient of the pressure appears in \eqref{eq:NS}, one can replace $\Pi$ with 
$\Pi-\overline{\Pi}_{B(1)}$ by subtracting an average in above criteria.
In \cite{[CKN]}, Caffarelli et al. proved $\dim_{H}(\mathcal{S})\leq 1$ in terms of $\eps$- regularity criterion at all scales \eqref{ckn2}. While  Robinson and  Sadowski used  $\eps$- regularity criterion at one scale \eqref{Lin}  to obtain $\dim_{B}(\mathcal{S})\leq5/3(\approx 1.67)$ in \cite{[RS1]}.
Here $\dim_{H}(S)$ and $\dim_{B}(S)$ denote the Hausdorff dimension and box dimension of the set $S$, respectively.
For the background of fractal dimension, we refer the reader to \cite{[Falconer]}.
The relationship between Hausdorff dimension and box dimension is that the former  is less than the latter.

In the past decade,  several authors have tried to lower the box dimension of potential interior singular points for suitable weak solutions to the 3D Navier--Stokes equations to $1$ (Kukavica  \cite{[Kukavica]} $(135/82\approx 1.65)$; Kukavica and  Pei \cite{[KP]} $45/29(\approx1.55)$; Koh and Yang \cite{[KY16]} $95/63( 1.51)$; Wang and Wu $(360/277\approx1.30)$; Ren, Wang and Wu \cite{[RWW]} $(975/758\approx1.29)$; He, Wang and Zhou  \cite{[HWZ]} $(2400/1903\approx1.261)$; Wang  and Yang \cite{WY} $(7/6\approx 1.167)$). 

In this paper, making use of coarse-to-refined estimates based on mean-fluctuation decomposition, we further slightly lower the upper bound to $15/13\approx 1.154$.

\begin{definition}(Upper box dimension)
Define the parabolic distance by
\begin{equation}\label{eq:parabolic-distance}
d_{\mathrm p}\bigl((x,t),(y,s)\bigr)
:=\max\bigl\{|x-y|,|t-s|^{1/2}\bigr\}.
\end{equation}
For a compact set $X$, let $N_{\mathrm p}(X,r)$ be the least number of
parabolic balls of radius $r$ needed to cover $X$, and define the upper box dimension of $X$ as
\begin{equation}\label{eq:box-dim}
\pdim X
:=\limsup_{r\to 0}
\frac{\log N_{\mathrm p}(X,r)}{-\log r}.
\end{equation}
\end{definition}

\begin{theorem}(Local $\eps$-criterion at one scale)\label{thm:density}	
	There exist an absolute constant $\eps>0$ and a radius $\rho_0<1$ with the following property. If
	$Q_{4\rho}(z_0)\Subset\Omega\times(-T,0)$ with $\rho\le \rho_0$ and
	\begin{equation}\label{eq:density-criterion}
		\begin{aligned}
			\iint_{Q_{4\rho}(z_0)}\Bigl(
			|\nabla u|^2+|u|^{10/3}
			+|\pi-\avg\pi_{B_{4\rho}}(t)|^{5/3}
			+|\nabla\pi|^{5/4}
			\Bigr)\dd x\dd t
			\le \eps\rho^{15/13},
		\end{aligned}
	\end{equation}
	then $z_0$ is a regular point. Here the pressure mean is taken only in space:
	\[
	\avg\pi_{B_r}(t):=\frac1{|B_r|}\int_{B_r}\pi(x,t)\dd x.
	\]
\end{theorem}

Based on Theorem \eqref{thm:density}, we can prove the following involving the upper box dimension on interior singular point set.
\begin{theorem}(Main theorem)\label{thm:main}
Let $(u,\pi)$ be a suitable weak solution of \eqref{eq:NS},
then we have
\begin{equation}\label{eq:main-result}
\pdim(\mathcal{S})\le\frac{15}{13}.
\end{equation}
\end{theorem}

The rest of this paper is organized as follows.
In Section 2,  we introduce relevant notation and definitions of suitable weak solutions to \eqref{eq:NS} and  present a few auxiliary lemmas which are used in the proof of main theorem.
In Section 3, we prove the key estimate for $A(\rho)$ in terms of coarse-to-refined  argument based on mean-fluctuation decomposition, which enable us further improve Wang and Yang's results in \cite{WY}.
In Section 4, we prove Theorem \ref{thm:density} and Theorem \ref{thm:main}.
Throughout this paper, the generic positive constants $C$ may be different from line to line.
\section{Preliminaries}\label{sec:prelim}

\begin{definition} (Suitable weak solution)

A  pair   $(u, \pi)$  is called a suitable weak solution to the Navier--Stokes equations \eqref{eq:NS} if the following conditions are satisfied:
\begin{enumerate}
\item \label{eq:energy-class}
$u\in L^\infty_tL^2_x (\Omega_T)\cap L^2_tH^1_x  (\Omega_T),
\quad \pi\in L^{3/2}_{t,x} (\Omega_T)$;
\item
$(u, \pi)$~solves \eqref{eq:NS} in $\mathbb{R}^{3}\times (-T,\,0) $ in the sense of distributions;
\item
For every nonnegative $\phi\in C_0^\infty(\Omega\times(-T,0))$ and
almost every $t$,
\begin{equation}\label{eq:LEI}
\begin{aligned}
&\int_\Omega |u(x,t)|^2\phi(x,t)\dd x
+2\int_{-T}^{t}\!\int_\Omega |\nabla u|^2\phi\dd x\dd s\\
&\quad\le
\int_{-T}^{t}\!\int_\Omega |u|^2(\partial_s\phi+\Delta\phi)\dd x\dd s
+\int_{-T}^{t}\!\int_\Omega (|u|^2+2\pi)
u\cdot\nabla\phi\dd x\dd s.
\end{aligned}
\end{equation}
\end{enumerate}
\end{definition}

\begin{remark}
From interpolation inequality and the relation between $u$ and $\pi$, we know following quantities
\begin{equation}\label{eq:integrability}
	|u|^{10/3},\quad |\nabla u|^2,\quad
	|\pi|^{5/3},\quad |\nabla\pi|^{5/4}
\end{equation}
all belong to $ L^1(\Omega_T)$.
\end{remark}

Without loss of generality, we assume $z_0=(0,0)$, since N-S equations \eqref{eq:NS} is invariant under translation.
Define
\begin{equation}\label{eq:functionals}
\begin{aligned}
A(r)&:=r^{-1}\esssup_{t_0-r^2<t<t_0}
       \int_{B_r}|u(x,t)|^2\dd x,\\
E(r)&:=r^{-1}\iint_{Q_r}|\nabla u|^2\dd x\dd t,\\
E_p(r)&:=r^{p-5}\iint_{Q_r}|u|^p\dd x\dd t,\\
P_{5/4}(\nabla\pi,r)&:=r^{-1}
  \|\nabla\pi\|_{L^{5/4}(Q_r)},\\
P_{1,3/2}(\nabla\pi,r)&:=r^{-1}
  \|\nabla\pi\|_{L^{1,3/2}(Q_r)}.
\end{aligned}
\end{equation}
These quantities are invariant under the Navier--Stokes scaling
\[
u_\lambda(x,t)=\lambda u(\lambda x,\lambda^2t),
\qquad
\pi_\lambda(x,t)=\lambda^2\pi(\lambda x,\lambda^2t).
\]

The following are some lemmas used in the proof of Theorem \ref{thm:density}, most of which can be found in \cite{WY}.
\begin{lemma}(\cite{WY}, Scale interpolation for the velocity)\label{lem:velocity-interpolation}
If $0<r\le\rho/2$, then
\begin{equation}\label{eq:E2-interpolation}
E_2(r)
\le C\left(\frac{\rho}{r}\right)E(\rho) 
+C\Bigl(\frac r\rho\Bigr)^2A(\rho).
\end{equation}
\end{lemma}

\begin{lemma}(\cite{WY}, Pressure-gradient decay)\label{lem:pressure-decay}
For $0<r\le\rho/2$,
\begin{equation}\label{eq:pressure-decay}
P_{1,3/2}(\nabla\pi,r)
\le C \left( \frac\rho r\right)  E(\rho)
+C\left( \frac r\rho\right)  P_{5/4}(\nabla\pi,\rho).
\end{equation}
\end{lemma}

\begin{lemma}(\cite{WY}, $\eps$-regularity)\label{lem:epsilon-regularity}
There exists $\eps_0>0$ such that if, for some $r$,
\begin{equation}\label{eq:epsilon-regularity}
E_2(r)+E(r)+P_{1,3/2}(\nabla\pi,r)
\le\eps_0,
\end{equation}
then the center of the cylinder is a regular point.
\end{lemma}

\begin{lemma}(Nested-ball averages of pressure control)\label{lem:nested-average}
For $q\geq 1$,
\begin{equation}\label{eq:nested-average}
\int_{B_{2\rho}}|\pi-\avg\pi_{B_{2\rho}}|^q\dd x
\le C_q\int_{B_{2\rho}}|\pi-\avg\pi_{B_{4\rho}}|^q\dd x.
\end{equation}
\end{lemma}

\begin{proof}
First write
\begin{equation}\label{eq:average-difference}
\avg\pi_{B_{2\rho}}-\avg\pi_{B_{4\rho}}
=\frac1{|B_{2\rho}|}\int_{B_{2\rho}}
  (\pi-\avg\pi_{B_{4\rho}})\dd x.
\end{equation}
Jensen's inequality gives
\begin{equation}\label{eq:jensen-average}
|B_{2\rho}|\,
|\avg\pi_{B_{2\rho}}-\avg\pi_{B_{4\rho}}|^q
\le\int_{B_{2\rho}}|\pi-\avg\pi_{B_{4\rho}}|^q\dd x.
\end{equation}
The $L^q$ triangle inequality now proves \eqref{eq:nested-average}.
\end{proof}

\section{From coarse-scale to refined-scale estimates}\label{sec:coarse}

After translating the center, assume $z_0=(0,0)$.  Suppose that
\begin{equation}\label{eq:H-small}
\begin{aligned}
\iint_{Q_{4\rho}}\Bigl(
|\nabla u|^2+|u|^{10/3}
+|\pi-\avg\pi_{B_{4\rho}}|^{5/3}
+|\nabla\pi|^{5/4}\Bigr)\dd x\dd t
\le\eps\rho^{15/13},
\end{aligned}
\end{equation}
where $0<\eps<1$ will later be chosen sufficiently small.

\subsection{Coarse estimates from   $Q(4\rho)$ to  $Q(2\rho)$}

In the local energy inequality, choose a cutoff supported in $Q_{4\rho}$ and
equal to one on $Q_{2\rho}$. Like in \cite{WY}, using $\nabla\cdot u=0$ and estimating the
velocity and pressure terms by H\"older, Gagliardo--Nirenberg, and
Poincar\'e--Sobolev inequalities gives
\begin{equation}\label{eq:coarse-A-start}
\begin{aligned}
2\rho A(2\rho)
&\le C\Bigl(\iint_{Q_{4\rho}}|u|^{10/3}\Bigr)^{3/5}\\
&\quad+C\rho^{-1/2}
 \|u\|_{L^{10/3}(Q_{4\rho})}^{5/2}
 \|\nabla u\|_{L^2(Q_{4\rho})}^{1/2}\\
&\quad+C\rho^{-1/2}
 \|u\|_{L^{10/3}(Q_{4\rho})}
 \|\pi-\avg\pi_{B_{4\rho}}\|_{L^{5/3}(Q_{4\rho})}^{1/2}
 \|\nabla\pi\|_{L^{5/4}(Q_{4\rho})}^{1/2}.
\end{aligned}
\end{equation}
By \eqref{eq:H-small}, the three terms on the right are bounded respectively by
\begin{equation}\label{eq:coarse-A-three}
C\eps^{3/5}\rho^{9/13},
\qquad
C\eps\rho^{17/26},
\qquad
C\eps\rho^{17/26}.
\end{equation}
Since
 $0<\rho,\eps<1$, we obtain
\begin{equation}\label{eq:coarse-A}
A(2\rho)\le C\eps^{3/5}\rho^{-9/26}.
\end{equation}
On the other hand, \eqref{eq:H-small} directly gives
\begin{equation}\label{eq:coarse-E}
E(2\rho)\le C\eps\rho^{2/13}.
\end{equation}

\subsection{Refined estimates  from   $Q(2\rho)$ to  $Q(\rho)$ in terms of  mean-fluctuation decomposition}\label{sec:flux}

Let $\phi$ be supported in $Q_{2\rho}$, equal to one on $Q_\rho$, and satisfy
\begin{equation}\label{eq:cutoff}
|\partial_t\phi|+|\Delta\phi|\le C\rho^{-2},
\qquad |\nabla\phi|\le C\rho^{-1}.
\end{equation}
The local energy inequality becomes
\begin{equation}\label{eq:LEI-flux}
\rho\bigl(A(\rho)+E(\rho)\bigr)
\le C\bigl(I_1+I_2+I_3\bigr),
\end{equation}
where
\begin{equation}\label{eq:flux-definitions}
\begin{aligned}
I_1&:=\rho^{-2}\iint_{Q_{2\rho}}|u|^2\dd x\dd t,\\
I_2&:=\left|\iint_{Q_{2\rho}}|u|^2u\cdot\nabla\phi\dd x\dd t\right|,\\
I_3&:=\left|\iint_{Q_{2\rho}}
(\pi-\avg\pi_{B_{2\rho}})
u\cdot\nabla\phi\dd x\dd t\right|.
\end{aligned}
\end{equation}

By H\"older's inequality and $|Q_{2\rho}|\simeq\rho^5$,
\begin{equation}\label{eq:I0}
\begin{aligned}
I_1
&\le C\rho^{-2}|Q_{2\rho}|^{2/5}
\Bigl(\iint_{Q_{2\rho}}|u|^{10/3}\dd x\dd t\Bigr)^{3/5}\\
&\le C\eps^{3/5}\rho^{9/13}.
\end{aligned}
\end{equation}

As for the cubic velocity flux term $I_2$, we use the following mean-fluctuation decomposition.

Set
\begin{equation}\label{eq:mean-fluctuation}
m(t):=\avg u_{B_{2\rho}}(t),
\qquad
v(x,t):=u(x,t)-m(t).
\end{equation}
Here,
\[
\avg u_{B_{2\rho}}(t):=\frac1{|B_{2\rho}|}\int_{B_{2\rho}} u \dd x.
\]
Then $\avg v_{B_{2\rho}}=0$ and $\nabla v=\nabla u$. Since $m(t)$ is
independent of $x$ and $\phi$ vanishes near the spatial boundary, we have
\begin{equation}\label{eq:constant-flux-zero}
\int_{B_{2\rho}}|m(t)|^2m(t)\cdot\nabla\phi\dd x=0.
\end{equation}
We may therefore subtract the constant-vector flux $|m|^2m$ in the term $I_2$ , then which yield 
\begin{equation}\label{eq:subtract-constant-flux}
I_2\le\iint_{Q_{2\rho}}
\bigl||u|^2u-|m|^2m\bigr|\,|\nabla\phi|\dd x\dd t.
\end{equation}
Expanding $u=m+v$ gives the pointwise bound
\begin{equation}\label{eq:cubic-pointwise}
\bigl||u|^2u-|m|^2m\bigr|
\le C\bigl(|v|^3+|m||v|^2+|m|^2|v|\bigr).
\end{equation}

By  the minimizing property of the integral mean,
H\"older's inequality, and Poincar\'e inequality:

\begin{equation}\label{eq:v-m-basic}
\begin{aligned}
\|v\|_{L_t^\infty L_x^2(Q_{2\rho})}^2
&\le C\rho A(2\rho),\\
\|m\|_{L_t^\infty(-4\rho^2,0)}^2
&\le C\rho^{-2}A(2\rho),\\
\|v\|_{L^2(Q_{2\rho})}^2
&\le C\rho^3E(2\rho).
\end{aligned}
\end{equation}
The last estimate is obtained explicitly as follows:
\[
\begin{aligned}
\iint_{Q_{2\rho}}|v|^2\dd x\dd t
&\le C\rho^2\iint_{Q_{2\rho}}|\nabla v|^2\dd x\dd t\\
&=C\rho^2\iint_{Q_{2\rho}}|\nabla u|^2\dd x\dd t
\le C\rho^3E(2\rho).
\end{aligned}
\]

The Gagliardo--Nirenberg  inequalities yield
\begin{equation}\label{eq:v-L3}
\|v(t)\|_{L^3(B_{2\rho})}^3
\le C\|v(t)\|_{L^2(B_{2\rho})}^{3/2}\|\nabla v(t)\|_{L^2(B_{2\rho})}^{3/2}.
\end{equation}
Applying H\"older's inequality in time, we obtain
\begin{equation}\label{eq:N1-detail}
\begin{aligned}
&\rho^{-1}\iint_{Q_{2\rho}}|v|^3\dd x\dd t\\
&\le C\rho^{-1}
\Bigl(\esssup_{-4\rho^2\le t\le 0}\|v(t)\|_{L^2(B_{2\rho})}^2\Bigr)^{3/4}
\Bigl(\int_{-4\rho^2}^{0}\|\nabla v(t)\|_{L^2(B_{2\rho})}^2\dd t\Bigr)^{3/4}
(4\rho^2)^{1/4}\\
&\le C\rho A(2\rho)^{3/4}E(2\rho)^{3/4}.
\end{aligned}
\end{equation}

By \eqref{eq:v-m-basic},
\begin{equation}\label{eq:N2-detail}
\begin{aligned}
&\rho^{-1}\iint_{Q_{2\rho}}|m||v|^2\dd x\dd t\\
&\le \rho^{-1}\|m\|_{L_t^\infty(I_{2\rho})}
\iint_{Q_{2\rho}}|v|^2\dd x\dd t\\
&\le C\rho A(2\rho)^{1/2}E(2\rho),
\end{aligned}
\end{equation}
and
\begin{equation}\label{eq:N3-detail}
\begin{aligned}
&\rho^{-1}\iint_{Q_{2\rho}}|m|^2|v|\dd x\dd t\\
&\le \rho^{-1}\|m\|_{L_t^\infty(I_{2\rho})}^2
|Q_{2\rho}|^{1/2}\|v\|_{L^2(Q_{2\rho})}\\
&\le C\rho A(2\rho)E(2\rho)^{1/2}.
\end{aligned}
\end{equation}
Consequently,
\begin{equation}\label{eq:In-AE}
I_2\le C\rho\Bigl[
A(2\rho)^{3/4}E(2\rho)^{3/4}
+A(2\rho)^{1/2}E(2\rho)
+A(2\rho)E(2\rho)^{1/2}
\Bigr].
\end{equation}

Substituting \eqref{eq:coarse-A}--\eqref{eq:coarse-E} into \eqref{eq:In-AE}, the three terms satisfy
\begin{equation}\label{eq:In-three-powers}
\begin{aligned}
\rho A^{3/4}(2\rho)E^{3/4}(2\rho)
&\le C\eps^{6/5}\rho^{89/104},\\
\rho A^{1/2}(2\rho)E(2\rho)
&\le C\eps^{13/10}\rho^{51/52},\\
\rho A(2\rho)E^{1/2}(2\rho)
&\le C\eps^{11/10}\rho^{19/26}.
\end{aligned}
\end{equation}
For $0<\rho,\eps<1$, thus
\begin{equation}\label{eq:In-final}
I_2\le C\eps^{11/10}\rho^{19/26}.
\end{equation}

By Lemma \ref{lem:nested-average}, the  $L^{5/3}$ norm of  $\pi-\avg\pi_{B_{4\rho}}$ on
$B_{4\rho}$ in \eqref{eq:H-small} also controls  $L^{5/3}$ norm of  $\pi-\avg\pi_{B_{2\rho}}$ on
$Q_{2\rho}$. Using $u=v+m$, split
\begin{equation}\label{eq:Ipi-split}
I_3\le I_{3 v}+I_{3 m},
\end{equation}
where
\begin{equation}\label{eq:Ipi-parts}
I_{3 v}:=\left|\iint_{Q_{2\rho}}(\pi-\avg\pi_{B_{2\rho}})\,v\cdot\nabla\phi\dd x\dd t\right|,
\qquad
I_{3 m}:=\left|\iint_{Q_{2\rho}}(\pi-\avg\pi_{B_{2\rho}}) \,m\cdot\nabla\phi\dd x\dd t\right|.
\end{equation}

By H\"older's inequality in time and \eqref{eq:v-m-basic},
\begin{equation}\label{eq:v-L4L2}
\begin{aligned}
\|v\|_{L_t^4L_x^2(Q_{2\rho})}
&\le (4\rho^2)^{1/4}
\|v\|_{L_t^\infty L_x^2(Q_{2\rho})}\\
&\le C\rho A(2\rho)^{1/2}.
\end{aligned}
\end{equation}
 Observe that
\begin{equation}\label{eq:pressure-interpolation-indices}
\frac12=\frac14\cdot\frac35
+\frac34\cdot\frac7{15},
\qquad
\frac7{15}=\frac45-\frac13,
\end{equation}
  H\"older's and Poincar\'e--Sobolev inequality therefore give
\begin{equation}\label{eq:a-L2}
\|(\pi-\avg\pi_{B_{2\rho}})(t)\|_{L^2(B_{2\rho})}
\le C\|(\pi-\avg\pi_{B_{2\rho}})(t)\|_{L^{5/3}}^{1/4}
\|\nabla\pi(t)\|_{L^{5/4}}^{3/4}.
\end{equation}
Integrating \eqref{eq:a-L2} in time from $-4\rho^2$ to $0$ with respect to $L^{4/3}$, then exploiting  H\"older's inequality. Explicitly,
\begin{equation}\label{eq:a-time-detail}
\begin{aligned}
\|(\pi-\avg\pi_{B_{2\rho}})\|_{L_t^{4/3}L_x^2}^{4/3}
&\le C\int_{-4\rho^2}^{0}
\|(\pi-\avg\pi_{B_{2\rho}})(t)\|_{L^{5/3}(B_{2\rho})}^{1/3}
\|\nabla\pi(t)\|_{L^{5/3}(B_{2\rho})}\dd t\\
&\le C
\Bigl(\int_{-4\rho^2}^{0}\|(\pi-\avg\pi_{B_{2\rho}})(t)\|_{L^{5/3}(B_{2\rho})}^{5/3}\dd t\Bigr)^{1/5}
\Bigl(\int_{-4\rho^2}^{0}\|\nabla\pi(t)\|_{L^{5/3}(B_{2\rho})}^{5/4}\dd t\Bigr)^{4/5}.
\end{aligned}
\end{equation}
Thus
\begin{equation}\label{eq:a-L43L2}
\begin{aligned}
\|(\pi-\avg\pi_{B_{2\rho}})\|_{L_t^{4/3}L_x^2(Q_{2\rho})}
&\le C\|(\pi-\avg\pi_{B_{2\rho}})\|_{L^{5/3}(Q_{2\rho})}^{1/4}
\|\nabla\pi\|_{L^{5/4}(Q_{2\rho})}^{3/4}\\
&\le C\eps^{3/4}\rho^{45/52}.
\end{aligned}
\end{equation}
Combining \eqref{eq:v-L4L2}, \eqref{eq:a-L43L2}, and
$|\nabla\phi|\le C/\rho$ yields
\begin{equation}\label{eq:Ipiv}
\begin{aligned}
I_{3 v}
&\le C\rho^{-1}
\|(\pi-\avg\pi_{B_{2\rho}})\|_{L_t^{4/3}L_x^2}
\|v\|_{L_t^4L_x^2}\\
&\le C\eps^{3/4}
\rho^{45/52}A(2\rho)^{1/2}\\
&\le C\eps^{21/20}\rho^{9/13}.
\end{aligned}
\end{equation}

Because $m(t)$ is constant in space and $\phi$ vanishes near the boundary,
integration by parts gives
\begin{equation}\label{eq:Ipm-integration-by-parts}
\int_{B_{2\rho}}(\pi-\avg\pi_{B_{2\rho}})\,m\cdot\nabla\phi\dd x
=-\int_{B_{2\rho}}\phi\,m\cdot\nabla\pi\dd x.
\end{equation}
Therefore
\begin{equation}\label{eq:Ipm-start}
I_{3 m}
\le C\|m\|_{L_t^\infty(I_{2\rho})}
\iint_{Q_{2\rho}}|\nabla\pi|\dd x\dd t.
\end{equation}
Space-time H\"older's inequality and \eqref{eq:H-small} imply
\begin{equation}\label{eq:grad-pi-L1}
\begin{aligned}
\iint_{Q_{2\rho}}|\nabla\pi|\dd x\dd t
&\le |Q_{2\rho}|^{1/5}
\|\nabla\pi\|_{L^{5/4}(Q_{2\rho})}\\
&\le C\eps^{4/5}\rho^{25/13}.
\end{aligned}
\end{equation}
Combining $\|m\|_{L_t^\infty(I_{2\rho})}\le C\rho^{-1}A(2\rho)^{1/2}$ with
\eqref{eq:coarse-A}, we get
\begin{equation}\label{eq:Ipm-final}
I_{3 m}
\le C\eps^{11/10}\rho^{3/4}.
\end{equation}

Hence \eqref{eq:Ipiv} and \eqref{eq:Ipm-final} give
\begin{equation}\label{eq:Ipi-final}
I_3\le C\eps^{21/20}\rho^{9/13}.
\end{equation}

Substitute \eqref{eq:I0}, \eqref{eq:In-final}, and \eqref{eq:Ipi-final}
into \eqref{eq:LEI-flux} and divide by $\rho$:
\begin{equation}\label{eq:A-three-terms}
\begin{aligned}
A(\rho)&\le C\Bigl(
\eps^{3/5}\rho^{-4/13}
+\eps^{11/10}\rho^{-7/26}
+\eps^{21/20}\rho^{-4/13}
\Bigr)\\
&\le C\eps^{3/5}\rho^{-4/13},
\end{aligned}
\end{equation}

\section{Proofs of Theorem \ref*{thm:density} and \ref*{thm:main}}\label{sec:closing}
{\bf Proof of Theorem \ref*{thm:density}}.

Choosing $\rho_0>0$ sufficiently small.
Set
\begin{equation}\label{eq:theta-beta}
\theta=\rho^\beta, \qquad \beta =\frac2{13}, 
\qquad 0<\theta<\frac18.
\end{equation}

By Lemma \ref{lem:velocity-interpolation} and the elementary estimate
\[
E(\theta\rho)\le\theta^{-1}E(\rho),
\]
we have
\begin{equation}\label{eq:E2E-small-scale}
E_2(\theta\rho)+E(\theta\rho)
\le C\bigl[\theta^{-1}E(\rho)+\theta^2A(\rho)\bigr].
\end{equation}

Combining\eqref{eq:coarse-E} with \eqref{eq:A-three-terms} imply
\begin{equation}\label{eq:first-branch-closing}
\begin{aligned}
E_2(\theta\rho)+E(\theta\rho)
&\le C\Bigl[
\eps\rho^{-\beta+\frac2{13}}
+\eps^{3/5}\rho^{2\beta-\frac4{13}}
\Bigr]\\
&\le C\eps^{3/5}.
\end{aligned}
\end{equation}

By \eqref{eq:H-small} and the definitions in \eqref{eq:functionals},
\begin{equation}\label{eq:P54-bound}
P_{5/4}(\nabla\pi,\rho)
\le C\eps^{4/5}\rho^{-1/13}.
\end{equation}
The Lemma \ref{lem:pressure-decay} gives
\begin{equation}\label{eq:pressure-closing}
\begin{aligned}
P_{1,3/2}(\nabla\pi,\theta\rho)
&\le C\theta^{-1}E(\rho)
+C\theta P_{5/4}(\nabla\pi,\rho)\\
&\le C\eps\rho^{-\beta+\frac2{13}}
+C\eps^{4/5}\rho^{\beta-\frac1{13}}\\
&\le C \eps^{4/5}
\end{aligned}
\end{equation}

 Choose $\eps>0$ sufficiently small such that the sum of \eqref{eq:first-branch-closing} and
\eqref{eq:pressure-closing} is less than the universal constant  $\eps_0$ in regularity criterion Lemma \ref{lem:epsilon-regularity}, hence the cylinder center is regular.
This proves Theorem \ref{thm:density}.

\vspace{3cm}

{\bf Proof of  Theorem \ref{thm:main}}.

By  Theorem \ref{thm:density}, for every
$z\in\mathcal{S}$ and every sufficiently small $r$,
\begin{equation}\label{eq:singular-lower-density}
\iint_{Q_{4r}(z)}\Bigl(
|\nabla u|^2+|u|^{10/3}
+|\pi-\avg\pi_{B_{4r}(z)}|^{5/3}
+|\nabla\pi|^{5/4}
\Bigr)\dd x\dd t
>\eps r^{15/13}.
\end{equation}

In light of known regularity results for suitable weak solutions to N-S equations, there exists 
a compact set  $K\subset\Omega_T$, such that $\mathcal{S}\subset K$.

For every  sufficiently small $\rho>0$, choose the maximal $\rho$-separated family
\[
\{z_j\}_{j=1}^{N_\rho}\subset\mathcal{S},
\]
such that the  parabolic cylinders $Q_{\rho/2}(z_j)$ may be chosen pairwise
disjoint, while $\cup_{j=1}^{N_\rho} Q_{\rho}(z_j)$  cover $\mathcal{S}$. Apply
\eqref{eq:singular-lower-density} with $r=\rho/8$; then
$Q_{4r}(z_j)=Q_{\rho/2}(z_j)$, so
\begin{equation}\label{eq:cover-sum}
\begin{aligned}
N_\rho\,\eps\Bigl(\frac\rho8\Bigr)^{15/13}
&\le\sum_{j=1}^{N_\rho}
\iint_{Q_{\rho/2}(z_j)} \Bigl(|\nabla u|^2+|u|^{10/3}
+|\pi-\avg\pi_{B_{4r}(z)}|^{5/3}
+|\nabla\pi|^{5/4}
\Bigr)\dd x\dd t\\
&\le \iint_{Q_T} \Bigl(|\nabla u|^2+|u|^{10/3}
+|\pi|^{5/3}
+|\nabla\pi|^{5/4}
\Bigr)\dd x\dd t <\infty.
\end{aligned}
\end{equation}
The right-hand side is controlled by disjointness of the cylinders and the
 integrability in \eqref{eq:integrability}. The pressure oscillation is
also controlled by Jensen's inequality:
\begin{equation}\label{eq:pressure-jensen-cover}
\int_B|\pi-\avg\pi_B|^{5/3}\dd x
\le C\int_B|\pi|^{5/3}\dd x.
\end{equation}
Thus,
\begin{equation}\label{eq:N-rho-bound}
N_\rho\le {C\over\eps} \rho^{-15/13}.
\end{equation}
By the definition \eqref{eq:box-dim} of upper box dimension,
\begin{equation}\label{eq:box-final}
\pdim(\mathcal{S})\le\frac{15}{13}.
\end{equation}
This proves Theorem \ref{thm:main}.

\end{document}